\documentclass[leqno,12pt]{amsart} 
\usepackage{amssymb}
\usepackage{amsmath}
\usepackage{amsthm}
\usepackage{amscd}
\usepackage{graphicx}
\usepackage[all]{xy}
\usepackage[dvips]{color}
\usepackage{mathtools}

\makeatletter
\newsavebox{\@brx}
\newcommand{\llangle}[1][]{\savebox{\@brx}{\(\m@th{#1\langle}\)}%
  \mathopen{\copy\@brx\kern-0.5\wd\@brx\usebox{\@brx}}}
\newcommand{\rrangle}[1][]{\savebox{\@brx}{\(\m@th{#1\rangle}\)}%
  \mathclose{\copy\@brx\kern-0.5\wd\@brx\usebox{\@brx}}}
\makeatother

\theoremstyle{plain} 
\newtheorem{theorem}{\indent\sc Theorem}[section]
\newtheorem{lemma}[theorem]{\indent\sc Lemma}
\newtheorem{corollary}[theorem]{\indent\sc Corollary}
\newtheorem{proposition}[theorem]{\indent\sc Proposition}

\theoremstyle{definition} 

\newtheorem{remark}[theorem]{\indent\sc Remark}

\newtheorem{question}[theorem]{\indent\sc Question}
\newcommand{\del}{\partial}
\newcommand{\delbar}{\overline{\partial}}
\newcommand{\ddbar}{\del\delbar}
\newcommand{\vp}{\varphi}
\newcommand{\ve}{\varepsilon}
\newcommand{\U}{\mathrm{U}(1)}

\begin{document}

\title[On the complement of a hypersurface]
{On the complement of a hypersurface with flat normal bundle which corresponds to a semipositive line bundle} 

\author[T. Koike]{Takayuki Koike} 

%
\address{
Department of Mathematics, Graduate School of Science, Osaka City University \endgraf
3-3-138, Sugimoto, Sumiyoshi-ku Osaka, 558-8585 \endgraf
Japan
}
\email{tkoike@sci.osaka-cu.ac.jp}

\maketitle

\begin{abstract}
We investigate the complex analytic structure of the complement of a non-singular hypersurface with unitary flat normal bundle when the corresponding line bundle admits a Hermitian metric with semipositive curvature. 
\end{abstract}

\section{Introduction}

Let $X$ be a complex manifold. 
We say that a line bundle $L$ on $X$ is {\it $C^r$-semipositive} if there exists a Hermitian metric $h$ on $L$ of class $C^r$ such that the Chern curvature tensor $\sqrt{-1}\Theta_h$ of $h$ is positive semi-definite at any point of $X$ ($r\in\{2, 3, 4, \dots\}\cup\{\infty\}$). 
Let $Y$ be a compact non-singular hypersurface of $X$. 
In this paper, we investigate the complex analytic structure of the complement $X\setminus Y$ and small neighborhoods of $Y$ when $[Y]$ is $C^r$-semipositive, 
where $[Y]$ is the line bundle on $X$ which corresponds to the effective divisor $Y$. 
We are interested in the case where the (holomorphic) normal bundle $N_{Y/X}$ of $Y$ is {\it unitary flat}: 
i.e. the transition functions of $N_{Y/X}$ are locally constant functions valued in $\U\coloneqq \{t\in \mathbb{C}\mid |t|=1\}$ for a suitable choice of the system of local trivializations. 
Note that a holomorphic line bundle on a compact K\"ahler manifold is unitary flat if it is topologically trivial (a theorem of Kashiwara, see \cite[\S 1]{U} for example). 

One main conclusion of the main results in this paper is the following: 

\begin{theorem}\label{thm:main}
Let $X$ be a compact connected K\"ahler manifold and $Y$ be a connected non-singular hypersurface of $X$ with topologically trivial normal bundle. 
Assume that $[Y]$ is $C^\infty$-semipositive. \\
$(i)$ When the line bundle $N_{Y/X}^m\coloneqq N_{Y/X}^{\otimes m}$ is holomorphically trivial for some positive integer $m$, there exists a holomorphic fibration $f\colon X\to B$ to a manifold $B$ such that $Y$ is a fiber of $f$. \\
$(ii)$ When the line bundle $N_{Y/X}^m$ is not holomorphically trivial for any positive integer $m$, 
the Hartogs type extension theorem holds on $X\setminus Y$: 
i.e. for any neighborhood $V$ of $Y$ and any holomorphic function $f\colon V\setminus Y\to \mathbb{C}$, there exists a holomorphic function $F\colon X\setminus Y\to \mathbb{C}$ such that $F|_{V\setminus Y}=f$. 
\end{theorem}

Theorem \ref{thm:main} $(i)$ holds not only when $[Y]$ is $C^\infty$-semipositive but also $[Y]$ is $C^r$-semipositive if $r\geq \max\{3, 2m\}$ (see Remark \ref{rmk:mainthm-1}). 
Theorem \ref{thm:main} $(i)$ is shown by combining \cite[Theorem 5.1]{N} (see also \cite[Theorem 1.3]{CLPT}) and the following: 

\begin{theorem}\label{thm:submain_1}
Let $X$ be a complex manifold and $Y$ be a non-singular compact hypersurface of $X$ with unitary flat normal bundle. 
Assume that $[Y]$ is $C^r$-semipositive. 
Then, for a positive integer $m$ with $\max\{3, 2m\}\leq r$, 
Ueda type of the pair $(Y, X)$ is larger than or equal to $m+1$: i.e. 
there exist open subsets $V_j$'s of $X$ which cover $Y$ and 
a holomorphic defining function $w_j$ of $V_j\cap Y$ on each $V_j$ such that 
$t_{jk}w_k=w_j+O(w_j^{m+2})$ holds on each $V_j\cap V_k$ for some constant $t_{jk}\in\U$. 
\end{theorem}

When $X$ is a surface, one can easily deduce Theorem \ref{thm:submain_1} from \cite[Theorem 1.1]{K2015}. 
Therefore, this theorem can be regarded as a generalization of \cite[Theorem 1.1]{K2015}. 
Note that this theorem is also a generalization of \cite[Theorem 1.6]{K2020}, which is on the case of $(r, m)=(\infty, 2)$. 
By combining Theorem \ref{thm:submain_1} and \cite[Theorem 3]{U}, one has the following: 

\begin{corollary}\label{cor:main}
Let $X$ be a complex manifold and $Y$ be a non-singular compact hypersurface of $X$ 
such that $N_{Y/X}^m$ is holomorphically trivial for some positive integer $m$. 
Assume that $[Y]$ is $C^\infty$-semipositive. 
Then there exists a neighborhood $V$ of $Y$ and a holomorphic function $f\colon V\to \mathbb{C}$ 
such that the pull-back of the divisor $\{0\}$ of $\mathbb{C}$ by $f$ coincides with $mY$. 
\end{corollary}

Theorem \ref{thm:submain_1} is shown by using the same strategy as that of the proof of \cite[Theorem 1.6]{K2020}: we investigate the coefficient functions obtained by considering the Taylor expansion of a local weight function of a Hermitian metric of $[Y]$, and use them suitably to improve the system of local defining functions of $Y$ inductively. 
As one needs to linearize the transitions of local defining functions in higher order jet, 
much more careful arguments are needed than that we used in \cite{K2020}. 

Theorem \ref{thm:main} $(ii)$ follows from the following: 

\begin{theorem}\label{thm:submain_2}
Let $X$ be a connected compact K\"ahler manifold 
and $Y$ be a non-singular connected hypersurface with topologically trivial normal bundle. 
Assume either of the following two conditions: 
$(a)$ The Hartogs type extension theorem does {\bf not} hold on $X\setminus Y$, or 
$(b)$ There exists a connected open neighborhood $\Omega$ of $Y$ in $X$ such that the boundary $\del\Omega$ is a locally pseudoconvex $C^2$-smooth real submanifold of real codimension $1$. 
Then the following are equivalent: \\
$(i)$ $[Y]$ is $C^\infty$-semipositive. \\
$(ii)$ There exists a neighborhood $V$ of $Y$ such that the line bundle $[Y]|_V$ is unitary flat. \\
Moreover, under the assumption $(b)$, $\del\Omega$ is a Levi-flat hypersurface if the conditions $(i)$ and $(ii)$ hold. 
\end{theorem}

Note that Corollary \ref{cor:main} and Theorem \ref{thm:submain_2} give a partial solution of Conjecture \cite[Conjecture 1.1]{K2020}, which is one of the biggest motivations of the present paper. 
In \cite{O2}, Ohsawa showed two types of Hartogs type extension theorems (see also \cite[\S 5.2.2]{O3}). 
When $Y$ is a smooth hypersurface of a compact connected K\"ahler manifold $X$ (for simplicity), \cite[Theorem 0.2]{O2} says that the Hartogs type extension theorem holds on $X\setminus Y$ if the normal bundle $N_{Y/X}$ admits a smooth Hermitian metric $h$ such that the Chern curvature $\sqrt{-1}\Theta_h$ is positive semi-definite and not identically zero. 
Theorem \ref{thm:main} $(ii)$ can be regarded as a variant of this theorem. 
Theorem \ref{thm:submain_2} is shown by using \cite[Theorem 0.1]{O2}, which is another type of the Hartogs type extension theorem (see also \S \ref{section:ohsawa_Hartogs_ext} here). 
In order to apply this theorem for proving Theorem \ref{thm:submain_2}, 
we construct a plurisubharmonic function $\psi$ on $X\setminus Y$ with logarithmic growth along $Y$ by assuming the condition $(i)$ of Theorem \ref{thm:submain_2}. 
We will construct such a function $\psi$ so that the difference between $\psi$ and the local weight function of a suitable Hermitian metric of $[Y]$ is pluriharmonic. 
We will apply \cite[Theorem 0.1]{O2} and our detailed expression of the local weight functions obtained in the proof of Theorem \ref{thm:submain_1}, and count the number of positive eigenvalues of the complex Hessian of $\psi$, which enables us to consider the Monge--Amp\`ere foliation. 
Theorem \ref{thm:submain_2} is shown by running a complex-dynamical argument for the holonomy map of this foliation such as the argument in \cite[\S 4]{KO}. Here we apply P\'erez-Marco's theory on the dynamics of the non-linearizable irrationally indifferent fixed point \cite{P}. 

By combining the arguments in the proofs of Theorem \ref{thm:submain_1} and Theorem \ref{thm:submain_2}, we also have the following result as a partial solution of Conjecture \cite[Conjecture 1.1]{K2020}. 
\begin{theorem}\label{thm:main_last}
Let $X$ be a complex manifold and $Y$ be a non-singular compact hypersurface of $X$  with unitary flat normal bundle. 
Then the following are equivalent: \\
$(i)$ There exists a neighborhood $V$ of $Y$ such that $[Y]|_V$ admits a real-analytic Hermitian metric with semipositive curvature. \\
$(ii)$ There exists a neighborhood $V$ of $Y$ such that $[Y]|_V$ is unitary flat. 
\end{theorem}

The organization of the paper is as follows. 
In \S 2, we explain some terms and collect some known results which we will use in this paper. 
In \S 3, we prove Theorem \ref{thm:submain_1} and Corollary \ref{cor:main}. 
In \S 4, we prove Theorem \ref{thm:submain_2}. 
In \S 5, we prove Theorem \ref{thm:main} and Theorem \ref{thm:main_last}. 
\vskip3mm
{\bf Acknowledgment. } 
The author would like to give heartfelt thanks to 
Professor Takeo Ohsawa for his comments and suggestions of inestimable value. 
He is also grateful to Professor Shin-ichi Matsumura for discussions on some results  related to Theorem \ref{thm:main} $(i)$. 
This work was supported by JSPS Grant-in-Aid for Early-Career Scientists 20K14313, 
by MEXT Grant-in-Aid for Leading Initiative for Excellent Young Researchers（LEADER) No. J171000201,  
and partly by Osaka City University Advanced Mathematical Institute (MEXT Joint Usage/Research Center on Mathematics and Theoretical Physics).

\section{Preliminary}

Let $X$ be a complex manifold and $Y$ be a compact non-singular hypersurface of $X$. 
Assume that the normal bundle $N_{Y/X}$ is unitary flat. 
In this section, we explain some terms and collect some known results which we will use in this paper. 

\subsection{Hermitian metrics of $[Y]$ and plurisubharmonic functions on $X\setminus Y$}

For $r\in \{2, 3, \dots\}\cup\{\infty\}$, 
we show the following proposition in this subsection. 

\begin{proposition}\label{prop:1}
Let $X$ be a complex manifold and $Y$ be a compact non-singular hypersurface of $X$ with unitary flat normal bundle. 
The following are equivalent: \\
$(i)$ $[Y]$ is $C^r$-semipositive. \\
$(ii)$ There exists a plurisubharmonic function $\psi\colon X\setminus Y\to\mathbb{R}$ of class $C^r$ which satisfies the following condition: 
For any point $p\in Y$, there exist a neighborhood $W$ of $p$ in $X$, 
a holomorphic defining function $w\colon W\to \mathbb{C}$ of $W\cap Y$, 
and a plurisubharmonic function $\vp\colon W\to \mathbb{R}$ of class $C^r$ such that 
$\psi=-\log |w|^2+\vp$ holds on $W\setminus Y$ and 
that the complex Hessian of $\vp$ has at least one positive eigenvalue at any point of $W$. 
\end{proposition}

\begin{proof}
First we show that the assertion $(ii)$ implies $(i)$. 
We denote by $\pi\colon [Y]\to X$ the projection. 
Take a function $\psi$ as in the assertion $(ii)$. 
Consider the Hermitian metric $h$ on $[Y]$ defined by 
\[
\langle \xi, \eta\rangle_{h, x}
\coloneqq e^{-\psi(x)}\cdot \frac{\xi}{f(x)}\cdot \overline{\left(\frac{\eta}{f(x)}\right)}, 
\]
where $x\in X$ is a point, $\xi, \eta\in[Y]_x\coloneqq \pi^{-1}(x)$, and $f\in H^0(X, [Y])=H^0(X, \mathcal{O}_X(Y))$ is a canonical section of $[Y]$: i.e. $f$ is a global holomorphic section of $[Y]$ whose zero divisor coincides with the divisor $Y$. 
Then it can be easily checked that $h$ is a Hermitian metric on $[Y]$ of class $C^r$ with semipositive curvature. 

Next we show that the assertion $(i)$ implies $(ii)$. 
Take a Hermitian metric $h$ on $[Y]$ of class $C^r$ with semipositive curvature. 
Using a canonical section $f\in H^0(X, [Y])$ of $[Y]$, we define a function $\Psi$ by letting 
$\Psi\coloneqq -\log |f|^2_h$. 
Note that this function $\Psi$ satisfies the property as in the assertion $(ii)$ except the condition that 
the complex Hessians have at least one positive eigenvalue. 

Consider the function $\widehat{\psi}\colon X\setminus Y\to \mathbb{R}$ defined by 
$\widehat{\psi} \coloneqq \Psi + e^{-\Psi}$. 
As the function $\chi\colon\mathbb{R}_{>0}\ni t\mapsto t+e^{-t}\in\mathbb{R}$ is increasing and convex, 
this function $\widehat{\psi}$ is plurisubharmonic on $V_1\coloneqq \{x\in X\setminus Y\mid \psi(x)>0\}$. 
It follows from a simple computation by using local coordinates that 
the complex Hessian of the local function $\widehat{\psi}$ has at least one positive eigenvalue on a neighborhood of any point of $Y$ in $X$. 
Take a function $\lambda\colon \mathbb{R}\to \mathbb{R}$ of class $C^\infty$ 
such that $\lambda$ is non-decreasing, convex, $\lambda|_{(-\infty, \ve]}\equiv 1/2$ holds for a positive number $\ve$, and that $\lambda(t)=t$ holds on the interval $[1, +\infty)$ (one can actually construct such a function $\lambda$ by taking {\it the regularized maximum} of the constant function $1/2$ and the identity function for example. See \cite[\S 5.E]{D}). 
Then it is easily observed that the function $\psi$ defined by 
\[
\psi(p) \coloneqq 
\begin{cases}
\lambda\circ\widehat{\psi}(p) & \text{if}\ p\in V_1\\
1/2 &  \text{if}\ p\in X\setminus V_1
\end{cases}
\]
enjoys the condition. 
\end{proof}

\begin{remark}\label{rmk:vp00}
Take open sets $V_j$'s of $X$ which cover $Y$. 
By shrinking $V_j$'s, we often consider local holomorphic defining functions $w_j$'s of $Y$ on $V_j$'s such that 
$dw_j=t_{jk}dw_k$ holds on each $V_j\cap V_k\cap Y$ for some constant $t_{jk}\in \U$. 
For a function $\psi$ as in Proposition \ref{prop:1}, 
the function $\vp_j\coloneqq  \psi + \log |w_j|^2$ can be regarded as the local weight function of a metric on $[Y]$ with semipositive curvature. 
As is easily observed, $\{(V_j\cap Y, \vp_j|_{V_j\cap Y})\}$ glues up to define a global plurisubharmonic function on a compact manifold $Y$. 
Thus one can deduce from the maximum principle that there exists a constant $A\in \mathbb{R}$ such that $\vp_j|_{V_j\cap Y}\equiv A$. 
Therefore, one may assume that $\vp_j|_{V_j\cap Y}\equiv 0$ by replacing $\psi$ with $\psi-A$. 
\end{remark}

\begin{remark}\label{rmk:omega_boundary}
It is also easily observed that, for a function $\psi$ as in Proposition \ref{prop:1}, 
it holds that the set $\{\psi>N\}\cup Y$ is a sufficiently small tubular neighborhood of $Y$ whose boundary is a real submanifold of class $C^r$ for a sufficiently large number $N$. 
\end{remark}

\subsection{Ueda's theory on a neighborhood of $Y$}\label{section:ueda_theory}

In \cite{U}, Ueda investigated the complex analytic structure on a neighborhood of $Y$. 
In this subsection, we give a brief explanation on this Ueda theory. 

Take a finite open covering $\{U_j\}$ of $Y$, a neighborhood $V_j$ of $U_j$ in $X$ with $V_j\cap Y=U_j$, and 
a defining function $w_j\colon V_j\to \mathbb{C}$ of $U_j$ for each $j$. 
In what follows, we always assume that $\{U_j\}$ and $\{V_j\}$ are fine enough so that $U_j$'s and $V_j$'s are simply connected and Stein, 
and that $U_{jk}\coloneqq U_j\cap U_k$ is empty if and only if $V_{jk}\coloneqq V_j\cap V_k$ is empty. 

As is simply observed, one may assume that $dw_j=t_{jk}dw_k$ holds on each $U_{jk}$ for some constant $t_{jk}\in\mathrm{U}(1)$ by changing $w_j$'s if necessary. 
We call such a system $\{(V_j, w_j)\}$ of local defining functions of $Y$ a system of {\it type $1$}. 
By shrinking $V_j$'s if necessary again, we assume that, for each $j$, there exists a holomorphic surjection ${\rm Pr}_{U_j}\colon V_j\to U_j$ such that the restriction ${\rm Pr}_{U_j}|_{U_j}$ is the identity and that $(w_j, z_j\circ {\rm Pr}_{U_j})$ are coordinates of $V_j$, where $z_j$ is a coordinate of $U_j$. 
In what follows, for any holomorphic function $f$ on $U_j$, we denote by the same letter $f$ the pull-back ${\rm Pr}_{U_j}^*f\coloneqq f\circ {\rm Pr}_{U_j}$. 
On $U_j$ and $U_k$ such that $U_{jk}\not=\emptyset$, 
one has the series expansion 
\[
t_{jk} w_k
=
w_j+\sum_{n= 2}^\infty f_{kj, n}(z_j) \cdot w_j^n, 
\]
where 
$f_{kj, n}$'s are holomorphic functions on $U_{jk}$ (we regard this also as a function defined by $({\rm Pr}_{U_j}|_{{\rm Pr}_{U_j}^{-1}(U_{jk})})^*f_{kj, n}$). 
For a positive integer $m$, we say that the system $\{(V_j, w_j)\}$ of local defining functions is of {\it type $m$} if $f_{kj, \ell}\equiv 0$ holds for any $\ell$ with $\ell\leq m$ and any $j, k$ with $U_{jk}\not=\emptyset$. 
If $\{(V_j, w_j)\}$ is of type $m$, it follows that 
$\left\{\left(U_{jk}, \ f_{kj, m+1} \right)\right\}$ 
satisfies the $1$-cocycle condition, 
and thus it defines an element of $H^1(Y, N_{Y/X}^{-m})$ (see \cite[\S 2]{U} for details). 
We denote this cohomology class by $u_m(Y, X)$, which is the definition of {\it $m$-th Ueda class}. 
Ueda class $u_m(Y, X)$ is well-defined up to the action of $\mathrm{U}(1)$ on $H^1(Y, N_{Y/X}^{-m})$, namely $[u_m(Y, X)]\in H^1(Y, N_{Y/X}^{-m})/\mathrm{U}(1)$ does not depend on the choice of the system of type $m$. 

From a simple observation, it follows that there exists a system of type $m+1$ if and only if $u_m(Y, X)=0$, 
and that one can also define $u_{m+1}(Y, X)$ if $u_m(Y, X)=0$. 
The pair $(Y, X)$ is said to be of {\it finite type} if, for some positive integer $n$, there exists a system of type $n$ such that $u_n(Y, X)\not=0$. 
In this case, $n$ is said to be {\it Ueda type} of the pair $(Y, X)$. 
Otherwise, the pair is said to be of {\it infinite type}. 
We say that Ueda type of the pair $(Y, X)$ is $\infty$ in this case. 

In \cite[\S 4]{U}, Ueda gave some sufficient conditions for the existence of a system $\{(V_j, w_j)\}$ such that $t_{jk}w_k=w_j$ holds on each $V_{jk}$ when the pair $(Y, X)$ is of infinite type. 
Note that $[Y]|_V$ is unitary flat for $V\coloneqq \bigcup_jV_j$ if and only if such a system exists. 
Note also that, if $N_{Y/X}^m$ is holomorphically trivial for some positive integer $m$, there exists a holomorphic function $f\colon V\to \mathbb{C}$ such that $f^*\{0\}=mY$ holds as divisors when such a system exists. 
Indeed, such $f$ can be constructed by letting $f|_{V_j}\coloneqq w_j^m$ ($w_j^m$'s glue up to each other since $t_{jk}^m=1$ holds for each $j$ and $k$ in this case). 
Though the following Ueda's theorem is stated only when $X$ is a surface in \cite{U}, this theorem and its proof in \cite[\S 4]{U} also work in general dimension (see also \cite{K2018}). 

\begin{theorem}[a part of {\cite[Theorem 3]{U}}]\label{thm:ueda_linearization}
Assume that the pair $(Y, X)$ is of infinite type and that there exists a positive integer $m$ such that $N_{Y/X}^m$ is holomorphically trivial. 
Then there exists a system $\{(V_j, w_j)\}$ such that $t_{jk}w_k=w_j$ holds on each $V_{jk}$. 
\end{theorem}

Note that there exist some examples of $(Y, X)$ of infinite type such that $[Y]|_V$ is not unitary flat for any neighborhood $V$ of $Y$ in the case where $N_{Y/X}^m$ is not holomorphically trivial for any integer $m$ (see \cite[\S 5]{U}, \cite{KO}). 

\subsection{Levi-flatness and Ohsawa's Hartogs type extension theorem for bounded domains in K\"ahler manifolds}\label{section:ohsawa_Hartogs_ext}

Let $\Omega$ be a relatively compact domain of $X$. 
Assume that $\Omega$ admits a plurisubharmonic defining function $F$ of class $C^\infty$ defined on a neighborhood of $\overline{\Omega}$: 
i.e. it holds that $\Omega=\{F<0\}$ and that $(dF)_p\not=0$ for each $p\in \del\Omega$. 
In this case, it is clear that $\del\Omega$ is a real submanifold of $X$ of class $C^\infty$ which is locally pseudoconvex. 
For each point $p\in \del\Omega$, denote by $T_{\del\Omega, p}^{1, 0}$ the subspace $\{v\in T_{X, p}\mid \del F(v)=0\}$ of the (holomorphic) tangent space $T_{X, p}$. 
The boundary $\del\Omega$ is said to be {\it Levi-flat} at $p\in\del\Omega$ if all the eigenvalues of the bilinear form on $T_{\del\Omega, p}^{1, 0}$ defined by restricting the complex Hessian of $F$ ({\it Levi form}) are zero. 
When $\del\Omega$ is Levi-flat, which means that $\del\Omega$ is Levi-flat at any point, the vectors in $T_{\del\Omega, p}^{1, 0}$'s define a foliation on $\del\Omega$, which is called {\it Levi foliation}. 

In the proof of Theorem \ref{thm:submain_2}, the following Ohsawa's Hartogs type extension theorem for bounded domains in K\"ahler manifolds plays an essential role: 

\begin{theorem}[{\cite[Theorem 0.1]{O2}}, see also {\cite[Theorem 5.11]{O3}}]\label{thm:ohsawa_Hartogs}
Let $X$ be a K\"ahler manifold and $\Omega\subset X$ be a relatively compact locally pseudoconvex domain such that the boundary $\del\Omega$ is a real hypersurface of class $C^2$. 
Assume that there exists a point of $\del\Omega$ at which $\del\Omega$ is not Levi-flat. 
Then the Hartogs type extension theorem holds on $\Omega$. 
\end{theorem}

Note that, when there exists a point $p\in\del\Omega$ at which the complex Hessian of $F$ has at least two positive eigenvalues, 
it clearly holds that the Levi form on $T_{\del\Omega, p}^{1, 0}$ has at least one positive eigenvalue. 
Therefore, $\del\Omega$ is not Levi-flat at the point $p$ in this case. 


\section{Proofs of Theorem \ref{thm:submain_1} and Corollary \ref{cor:main}}

\subsection{Proof of Theorem \ref{thm:submain_1}}

Here we show Theorem \ref{thm:submain_1}. 

\subsubsection{A fundamental fact and the outline of the proof}
In the proof of Theorem \ref{thm:submain_1}, the following fundamental lemma plays an important role. 

\begin{lemma}[{=\cite[Lemma 2.1]{K2020}}]\label{lem_positive_function}
Let $\Omega$ be a neighborhood of the origin in the complex plane $\mathbb{C}$ with the standard coordinate $w$, 
and $\Phi\colon \Omega\to \mathbb{R}$ be a function. 
Assume that $\Phi(w)\geq 0$ for any $w\in \Omega$ and that, for a positive integer $m$, $\Phi$ satisfies 
\[
\Phi(w) = \sum_{p=0}^m c_p\cdot w^p\overline{w}^{n-p} + o(|w|^m)
\]
as $|w|\to 0$, where $c_p$'s are complex constants. 
Then the following holds: \\
$(i)$ When $m$ is odd, $c_p= 0$ for any $p\in \{0, 1, 2, \dots, m\}$. \\
$(ii)$ When $m$ is even, the constant $c_{m/2}$ is a non-negative real number. Any of the other constants are zero if $c_{m/2}=0$. 
\end{lemma}

Let $X$ be a complex manifold of dimension $n$ and $Y$ be a non-singular compact hypersurface of $X$ with unitary flat normal bundle. 
Take an open covering $\{U_j\}$ of $Y$, open subsets $V_j$'s, 
coordinates $z_j=(z_j^1, z_j^2, \dots, z_j^d)$ of $U_j$ ($d\coloneqq n-1$), 
a holomorphic defining function $w_j\colon V_j\to \mathbb{C}$ of $U_j$, 
and constants $t_{jk}\in \U$ as in 
\S \ref{section:ueda_theory}. 
In what follows, 
we fix $\{U_j\}$ and $z_j$'s, 
whereas we change $w_j$'s and shrink $V_j$'s as necessary. 

Assume that $[Y]$ is $C^r$-semipositive ($r\geq 3$). 
Take a function $\psi$ as in Proposition \ref{prop:1}, 
and a plurisubharmonic function $\vp_j$ of class $C^r$ on $V_j$ such that 
$\psi(z_j, w_j) = -\log |w_j|^2 + \vp_j(z_j, w_j)$ holds on $V_j\setminus Y$. 
By Taylor's theorem, one has an expansion 
\[
\vp_j(z_j, w_j) = \sum_{p+q \leq r}\vp_j^{p, q}(z_j)\cdot w_j^p\overline{w_j}^q + \sum_{p+q= r}R_j^{p, q}(z_j, w_j)\cdot w_j^p\overline{w_j}^q
\]
of the function $\vp_j$, where $R_j^{p, q}$'s are functions which approach to $0$ as $|w_j|\to 0$ (in the expansion above, $p$ and $q$ run over the set $\mathbb{Z}_{\geq 0}\coloneqq\{\nu\in\mathbb{Z}\mid\nu\geq 0\}$). 
As is observed in Remark \ref{rmk:vp00}, we may assume that $\vp_j^{0, 0}\equiv 0$ for each $j$. 
Note that $\vp_j^{p, q}=\overline{\vp_j^{q, p}}$ holds, since $\vp_j$ is a real-valued function. 

In order to prove Theorem \ref{thm:submain_1}, we will show the following (Assertion)$_m$ for each $m\in\mathbb{Z}_{\geq 0}$ by induction. 
\begin{description}
\item[(Assertion)$_m$] {\it Assume that $r\geq \max\{3, 2m\}$ holds. 
Then there exists a system $\{(V_j, w_j)\}$ of type $m+1$. 
If the system $\{(V_j, w_j)\}$ is of type $m+1$, then the coefficient function $\vp_j^{p, q}$ is a constant function for $(p, q)=(0, m)$, $(p, q)=(m, 0)$, and for $(p, q)$ such that $0\leq p < m$ and $0\leq q < m$ hold. 
Moreover, if $r\geq 2m+2$ and the system $\{(V_j, w_j)\}$ is of type $m+1$, then $\vp_j^{p, q}$ is a constant function for any $(p, q)$ such that $0\leq p \leq m$ and $0\leq q \leq m$ hold. }
\end{description}

Note that (Assertion)$_0$ clearly holds, since $\vp_j^{0, 0}\equiv 0$. 
In what follows, we assume (Assertion)$_{m-1}$ and show that 
(Assertion)$_m$ holds for each $m\geq 1$. 

In order to show this, we first show that the coefficient function $\vp_j^{m, 0}$ is holomorphic in \S \ref{section:thm12_step1} when the system $\{(V_j, w_j)\}$ is of type $m$. 
Next, in \S \ref{section:thm12_step2}, we compute the expansion of the function $\vp_k-\vp_j$ in two-fold way and compare them to obtain equations on the difference between the functions $t_{jk}^{-p+q}\vp_k^{p, q}$ and $\vp_j^{p, q}$ on each $U_{jk}$. 
Combining these, we modify the defining function $w_j$ to construct a new system of type $m+1$ by using $\vp_j^{m, 0}$ in \S \ref{section:thm12_step3}. 
By using such a system, we complete the proof of (Assertion)$_m$ in \S \ref{section:thm12_step4}. 

In this proof, the following fundamental fact also plays an important role: 
on a compact complex manifold, any global pluriharmonic section of a unitary flat line bundle is locally constant. 
One can easily show this fact by using the maximum principle. 

\subsubsection{The holomorphy of the function $\vp_j^{m, 0}$}\label{section:thm12_step1}

Assume that $r\geq \max\{3, 2m\}$ holds. 
Note that $2m\leq r$ and $m+2\leq r$ hold, since $m\geq 1$. 

Take a system $\{(V_j, w_j)\}$ of type $m$. Let 
\[
t_{jk}w_k=w_j+\sum_{\nu=m+1}f_{kj, \nu}(z_j)\cdot w_j^\nu
\]
be the expansion on $V_{jk}$. 
As the function $(\vp_j)_{z_j^\lambda}\coloneqq \frac{\del}{\del z_j^\lambda}\vp_j$ is of class $C^{r-1}$, one has the expansion 
\[
(\vp_j)_{z_j^\lambda}(z_j, w_j) = 
\sum_{p+q \leq r-2}\vp_{\lambda, j}^{p, q}(z_j)\cdot w_j^p\overline{w_j}^q
  + \sum_{p+q= r-2}R_{\lambda, j}^{p, q}(z_j, w_j)\cdot w_j^p\overline{w_j}^q
\]
by Taylor's theorem, where 
$R_{\lambda, j}^{p, q}$'s are functions which approach to $0$ as $|w_j|\to 0$. 
Again by using the condition that $\vp_j$ is of class $C^{r}$ (and thus it is of class $C^{r-1}$ especially), one has that the function
\[
\vp_{\lambda, j}^{p, q} = \frac{1}{p!q!}\frac{\del^{p+q+1}}{\del w_j^p\del \overline{w_j}^q\del z_j^\lambda}\vp_j
\]
coincides with the partial derivative of the function 
\[
\vp_j^{p, q} = \frac{1}{p!q!}\frac{\del^{p+q}}{\del w_j^p\del \overline{w_j}^q}\vp_j
\]
with respect to the variable $z_j^\lambda$ for each $(p, q)\in (\mathbb{Z}_{\geq 0})^2$ with $p+q\leq r-2$. 
Therefore, 
\[
(\vp_j)_{z_j^\lambda}(z_j, w_j) = 
\sum_{p+q\leq r-2}(\vp_j^{p, q})_{z_j^\lambda}(z_j)\cdot w_j^p\overline{w_j}^q
  + \sum_{p+q= r-2}R_{\lambda, j}^{p, q}(z_j, w_j)\cdot w_j^p\overline{w_j}^q
\]
holds. 
As we are assuming (Assertion)$_{m-1}$, 
the function $\vp_j^{p, q}$ is constant if $p \leq m-1$ and $q \leq m-1$. 
Thus we have 
\begin{align*}
(\vp_j)_{z_j^\lambda}(z_j, w_j) =
&  (\vp_j^{m, 0})_{z_j^\lambda}(z_j)\cdot w_j^m + (\vp_j^{0, m})_{z_j^\lambda}(z_j)\cdot \overline{w_j}^m\\
&+\sum_{m+1\leq p+q\leq r-2}(\vp_j^{p, q})_{z_j^\lambda}(z_j)\cdot w_j^p\overline{w_j}^q
  + \sum_{p+q= r-2}R_{\lambda, j}^{p, q}(z_j, w_j)\cdot w_j^p\overline{w_j}^q, 
\end{align*}
from which it follows that 
\begin{equation}\label{eq:vp_zw}
(\vp_j)_{z_j^\lambda\overline{w_j}}(z_j, w_j) = 
m(\vp_j^{0, m})_{z_j^\lambda}(z_j)\cdot \overline{w_j}^{m-1}
  + o(|w_j|^{m-1})
\end{equation}
holds as $|w_j|\to 0$. 
Similarly, one has that 
\[
(\vp_j)_{z_j^\lambda \overline{z_j^\mu}}(z_j, w_j) = 
\begin{cases}
\sum_{p+q =1}(\vp_j^{p, q})_{z_j^\lambda\overline{z_j^\mu}}(z_j)\cdot w_j^p\overline{w_j}^q
 + o(|w_j|) & \text{if}\ m=1\\
\sum_{p+q \leq 2m-2}(\vp_j^{p, q})_{z_j^\lambda\overline{z_j^\mu}}(z_j)\cdot w_j^p\overline{w_j}^q
 + o(|w_j|^{2m-2}) & \text{if}\ m\geq 2
\end{cases}
\]
holds for $\lambda, \mu=1, 2, \dots, d$. 
Furthermore, it follows from the following Lemma \ref{lem_ph_inductively} that
\begin{equation}\label{eq:vp_zz}
(\vp_j)_{z_j^\lambda \overline{z_j^\mu}}(z_j, w_j) = 
\begin{cases}
o(|w_j|) & \text{if}\ m=1\\
o(|w_j|^{2m-2}) & \text{if}\ m\geq 2
\end{cases}
\end{equation}
holds. 

\begin{lemma}\label{lem_ph_inductively}
Let $M$ be a non-negative integer which is less than or equal to $\max\{1, 2m-2\}$. 
Then the function $\vp_j^{p, q}$ is pluriharmonic for any $(p, q)\in (\mathbb{Z}_{\geq 0})^2$ with $p+q=M$. 
\end{lemma}

\begin{proof}
Lemma clearly holds for $M=0$, since $\vp_j^{0, 0}\equiv 0$. 
Assume that the assertion holds for $M-1$. 
Then, for any element $\xi=(\xi^1, \xi^2, \dots, \xi^d)\in \mathbb{C}^d$, one has that
\[
\sum_{\lambda=1}^d\sum_{\mu=1}^d(\vp_j)_{z_j^\lambda \overline{z_j^\mu}}\cdot \xi^\lambda\cdot \overline{\xi^\mu} = 
\sum_{p+q = M}\left(\sum_{\lambda=1}^d\sum_{\mu=1}^d (\vp_j^{p, q})_{z_j^\lambda\overline{z_j^\mu}}\cdot \xi^\lambda\cdot \overline{\xi^\mu} \right)\cdot w_j^p\overline{w_j}^q
 + o(|w_j|^{M})
\]
holds. 
As $\vp_j$ is plurisubharmonic, this function is non-negative. 

When $M$ is odd, it follows from Lemma \ref{lem_positive_function} that 
\[
\sum_{\lambda=1}^d\sum_{\mu=1}^d (\vp_j^{p, q})_{z_j^\lambda\overline{z_j^\mu}}\cdot \xi^\lambda\cdot \overline{\xi^\mu}\equiv 0
\]
for $(p, q)$ with $p+q=M$. 
It easily follows from this that 
$2{\rm Re}\,\vp_j^{p, q} (= \vp_j^{p, q}+\vp_j^{q, p})$ and 
$2{\rm Im}\,\vp_j^{p, q} (= -\sqrt{-1}(\vp_j^{p, q}-\vp_j^{q, p}))$ 
satisfy the equations 
$\sum_{\lambda, \mu=1}^d ({\rm Re}\,\vp_j^{p, q})_{z_j^\lambda\overline{z_j^\mu}}\cdot \xi^\lambda\cdot \overline{\xi^\mu}\equiv 0$ and 
$\sum_{\lambda, \mu=1}^d ({\rm Im}\,\vp_j^{p, q})_{z_j^\lambda\overline{z_j^\mu}}\cdot \xi^\lambda\cdot \overline{\xi^\mu}\equiv 0$, 
which implies that both ${\rm Re}\,\vp_j^{p, q}$ and ${\rm Im}\,\vp_j^{p, q}$ are pluriharmonic. 
Therefore one has that 
$\vp_j^{p, q}$ is also pluriharmonic for $(p, q)$ with $p+q=M$. 

When $M$ is even, set $M=2L$. 
As $L\leq m-1$, (Assertion)$_{m-1}$ implies that $\vp_j^{L, L}$ is a constant function. 
Therefore the assertion follows from Lemma \ref{lem_positive_function} by using the same argument as in the case where $M$ is odd. 
\end{proof}

For $\lambda=1, 2, \dots, d$, 
denote by $H_{j, \lambda}(z_j, w_j)$ the Hermitian form on 
the space spanned by $\del/\del w_j$ and $\del/\del z_j^\lambda$ 
defined by restricting the complex Hessian of $\vp_j$. 
Then, by the equations (\ref{eq:vp_zw}) and (\ref{eq:vp_zz}), one has that 
\begin{align*}
{\rm det}\,H_{j, \lambda}
&= (\vp_j)_{w_j\overline{w_j}}\cdot (\vp_j)_{z_j^\lambda\overline{z_j^\lambda}}
-(\vp_j)_{z_j^\lambda\overline{w_j}}\cdot (\vp_j)_{w_j\overline{z_j^\lambda}}\\
&=-m^2\left|(\vp_j^{m, 0})_{\overline{z_j^\lambda}}\right|^2\cdot |w_j|^{2m-2}+o(|w_j|^{2m-2}). 
\end{align*}
As $\vp_j$ is plurisubharmonic, ${\rm det}\,H_{j, \lambda}$ is non-negative. 
Therefore one has that $(\vp_j^{m, 0})_{\overline{z_j^\lambda}}\equiv 0$ for each $\lambda$, which implies that the function $\vp_j^{m, 0}$ is holomorphic. 

\subsubsection{Expansions of $\vp_k-\vp_j$ and the gluing of $\vp_j^{p, q}$'s}\label{section:thm12_step2}

Here let us consider the expansion of the function $\vp_k(z_k, w_k)-\vp_j(z_j, w_j)$ in two-fold way. 
First, by using a holomorphic function $F_{kj, m+2}$ defined by $t_{jk}w_k=w_j+f_{kj, m+1}(z_j)\cdot w_j^{m+1}+w_j\cdot F_{kj, m+2}(z_j, w_j)$, one has 
\begin{align}\label{eq:diff_vp's_1}
\vp_k(z_k, w_k)-\vp_j(z_j, w_j) 
&=(\psi+\log|w_k|^2)-(\psi+\log|w_j|^2) \\
&=\log|w_k|^2-\log|w_j|^2 \nonumber \\
&=\log\left|1+f_{kj, m+1}(z_j)\cdot w_j^{m} + F_{kj, m+2}(z_j, w_j)\right|^2 \nonumber \\
&= f_{kj, m+1}\cdot w_j^m+\overline{f_{kj, m+1}}\cdot \overline{w_j}^m \nonumber \\
&\ \ \ \ - \frac{f_{kj, m+1}^2}{2}\cdot w_j^{2m} -\frac{\overline{f_{kj, m+1}^2}}{2}\cdot \overline{w_j}^{2m} \nonumber \\
&\ \ \ \ + F_{kj, m+2}+\overline{F_{kj, m+2}} +O(|w_j|^{2m+1}). \nonumber 
\end{align}

Let $G_{kj, m+2}^{(s)}$ be the holomorphic function defined by 
\[
(w_j + f_{kj, m+1}(z_j)\cdot w_j^{m+1} + \cdots)^s = w_j^s\cdot \left(1 + sf_{kj, m+1}(z_j)\cdot w_j^m + G_{kj, m+2}^{(s)}(z_j, w_j)\right)
\]
for each integer $s$. 
Then, as it holds that
\begin{align*}
&\vp_k^{p, q}(z_k)\cdot w_k^p\overline{w_k}^q\\
&=t_{jk}^{-p+q}\vp_k^{p, q}(z_k)\cdot (t_{jk}w_k)^p\overline{(t_{jk}w_k)}^q\\
&=t_{jk}^{-p+q}\vp_k^{p, q}(z_k)\cdot w_j^p\overline{w_j}^q\cdot (1 + pf_{kj, m+1}\cdot w_j^m + G_{kj, m+2}^{(p)})\cdot(1 + q\overline{f_{kj, m+1}(z_j)}\cdot \overline{w_j}^m + \overline{G_{kj, m+2}^{(q)}}),
\end{align*}
one has the second expansion 
\begin{align}\label{eq:diff_vp's_2}
&\vp_k(z_k, w_k)-\vp_j(z_j, w_j)\\
&= \sum_{p+q\leq 2m}\left(\vp_k^{p, q}(z_k)\cdot w_k^p\overline{w_k}^q-\vp_j^{p, q}(z_j)\cdot w_j^p\overline{w_j}^q\right)
+o(|w_j|^{2m}) \nonumber \\
&= \sum_{p+q\leq 2m}(t_{jk}^{-p+q}\vp_k^{p, q}(z_k)-\vp_j^{p, q}(z_j))\cdot w_j^p\overline{w_j}^q \nonumber  \\
&\ \ \ \ +\sum_{0<p+q\leq 2m}t_{jk}^{-p+q}\vp_k^{p, q}(z_k)\cdot \left(pf_{kj, m+1}\cdot w_j^{m+p}\overline{w_j}^q + q\overline{f_{kj, m+1}}\cdot w_j^p\overline{w_j}^{m+q}\right.  \nonumber \\
&\ \ \ \ \ \ \ \ \ \ \ \ \ \ \ \ \ \ \ \ \left.+ pq\cdot |f_{kj, m+1}|^2\cdot w_j^{m+p}\overline{w_j}^{m+q} + G_{kj, m+2}^{(p)}\cdot w_j^p\overline{w_j}^q+\overline{ G_{kj, m+2}^{(q)}}\cdot w_j^p\overline{w_j}^q\right) \nonumber \\
&\ \ \ \ +o(|w_j|^{2m}). \nonumber 
\end{align}

Note that $F_{kj, m+2}(z_j, w_j) = O(w_j^{m+1})$, 
$G_{kj, m+2}^{(s)}(z_j, w_j) = O(w_j^{m+1})$, and 
that $\vp_j^{p, q}$ is constant for each $(p, q)$ with $p \leq m-1$ and $q \leq m-1$ (by (Assertion)$_{m-1}$). 
Then it follows from these two expansions (\ref{eq:diff_vp's_1}) and (\ref{eq:diff_vp's_2}) that 
$t_{jk}^{-p+q}\vp_k^{p, q}-\vp_j^{p, q}\equiv 0$ holds for any $(p, q)\in (\mathbb{Z}_{\geq 0})^2$ with $p+q<m$ by the induction on $p+q$. 
Therefore, again by comparing two expansions (\ref{eq:diff_vp's_1}) and (\ref{eq:diff_vp's_2}), one has that 
\[
f_{kj, m+1}\cdot w_j^m+\overline{f_{kj, m+1}}\cdot \overline{w_j}^m
=\sum_{p+q= m}(t_{jk}^{-p+q}\vp_k^{p, q}(z_k)-\vp_j^{p, q}(z_j))\cdot w_j^p\overline{w_j}^q +o(|w_j|^{m}), 
\]
from which it follows that 
\begin{equation}\label{eq:diff_vp's_3}
t_{jk}^{-p+q}\vp_k^{p, q}-\vp_j^{p, q}=
\begin{cases}
\overline{f_{kj, m+1}} & \text{if}\ (p, q)=(0, m) \\
0 & \text{if}\ (p, q)=(1, m-1), (2, m-2), \dots, (m-1, 1) \\
f_{kj, m+1} & \text{if}\ (p, q)=(m, 0) 
\end{cases}
\end{equation}
holds on each $U_{jk}$. 

\subsubsection{Existence of a system of type $m+1$}\label{section:thm12_step3}

Set $\widehat{w}_j\coloneqq w_j-\vp_j^{m, 0}(z_j)\cdot w_j^{m+1}$. 
By shrinking $V_j$ if necessary, this function is also a defining function of $U_j$. 
As 
\begin{align*}
t_{jk}\widehat{w}_k
&= t_{jk}w_k-t_{jk}\vp_k^{m, 0}(z_k)\cdot w_k^{m+1} \\
&= (w_j + f_{kj, m+1}(z_j)\cdot w_j^{m+1} + O(w_j^{m+2})) - (t_{jk}^{-m}\vp_k^{m, 0}(z_k)\cdot w_j^{m+1} +O(w_j^{m+2}))\\
&= w_j - \vp_j^{m, 0}(z_j)\cdot w_j^{m+1} + O(w_j^{m+2}) 
= \widehat{w}_j + O(\widehat{w}_j^{m+2}), 
\end{align*}
one has that $\{(V_j, \widehat{w}_j)\}$ is a system of type $m+1$. 

In what follows, we replace $w_j$ with $\widehat{w}_j$ and assume that $\{(V_j, w_j)\}$ is a system of type $m+1$. 
Then it holds that $f_{kj, m+1}\equiv 0$. 
Thus it follows form the equation (\ref{eq:diff_vp's_3}) that 
$\{(U_j, \vp_j^{p, q}(z_j))\}$ glues up to define a global section of $N_{Y/X}^{-p+q}$ 
for each $(p, q)\in (\mathbb{Z}_{\geq 0})^2$ with $p+q= m$. 
As $m\leq \max\{1, 2m-2\}$, Lemma \ref{lem_ph_inductively} implies that 
$\{(U_j, \vp_j^{p, q}(z_j))\}$ is a plurisubharmonic global section of a unitary flat bundle $N_{Y/X}^{-p+q}$ for such a pair $(p, q)$, 
which means that $\vp_j^{p, q}$ is constant if $p+q= m$. 
Especially, both $\vp_j^{0, m}$ and $\vp_j^{m, 0}$ are constant functions. 

\subsubsection{End of the proof}\label{section:thm12_step4}
As in the previous subsection, we  assume that $\{(V_j, w_j)\}$ is a system of type $m+1$. 
Additionally, assume that $2m+2\leq r$. 
Then, as $f_{kj, m+1}\equiv 0$ and 
$\{(U_j, \vp_j^{p, q}(z_j))\}$ defines a locally constant global section of $N_{Y/X}^{-p+q}$ for any $(p, q)\in (\mathbb{Z}_{\geq 0})^2$ with $p+q \leq m$, 
it follows from the equations (\ref{eq:diff_vp's_1}) and (\ref{eq:diff_vp's_2}) that 
\begin{align}\label{eq:diff_vp's_4}
&F_{kj, m+2}+\overline{F_{kj, m+2}} \\
&= \sum_{m+1\leq p+q\leq 2m}(t_{jk}^{-p+q}\vp_k^{p, q}(z_k)-\vp_j^{p, q}(z_j))\cdot w_j^p\overline{w_j}^q \nonumber  \\
&\ \ \ \ +\sum_{0<p+q\leq 2m}t_{jk}^{-p+q}\vp_k^{p, q}(z_k)\cdot ( G_{kj, m+2}^{(p)}\cdot w_j^p\overline{w_j}^q+\overline{ G_{kj, m+2}^{(q)}}\cdot w_j^p\overline{w_j}^q) \nonumber \\
&\ \ \ \ +o(|w_j|^{2m}). \nonumber 
\end{align}
By using this, we have the following: 

\begin{lemma}\label{lem:12main}
For each $\mu=0, 1, 2, \dots, m$, $\{(U_j, \vp_j^{m, \mu})\}$ and $\{(U_j, \vp_j^{\mu, m})\}$ define locally constant global sections of $N_{Y/X}^{-m+\mu}$ and $N_{Y/X}^{m-\mu}$, respectively. 
\end{lemma}

Before giving a proof of Lemma \ref{lem:12main}, we note some implications of the arguments above. 
From Lemma \ref{lem_ph_inductively} and the arguments we made before Lemma \ref{lem_ph_inductively}, 
it follows that 
\begin{align*}
& \sum_{\lambda, \mu=1}^d(\vp_j)_{z_j^\lambda \overline{z_j^\mu}}\cdot \xi^\lambda\cdot \overline{\xi^\mu}  \\
&= \sum_{p+q = 2m-1}\left(\sum_{\lambda, \mu=1}^d(\vp_j^{p, q})_{z_j^\lambda\overline{z_j^\mu}}\cdot \xi^\lambda\cdot \overline{\xi^\mu} \right)\cdot w_j^p\overline{w_j}^q\\
&\ \ \ \ + \sum_{p+q = 2m}\left(\sum_{\lambda, \mu=1}^d(\vp_j^{p, q})_{z_j^\lambda\overline{z_j^\mu}}\cdot \xi^\lambda\cdot \overline{\xi^\mu} \right)\cdot w_j^p\overline{w_j}^q\\
&\ \ \ \  + o(|w_j|^{2m})
\end{align*}
holds for any $\xi=(\xi^1, \xi^2, \dots, \xi^d)\in\mathbb{C}^d$. 
Thus, by running the arguments as in the proof of Lemma \ref{lem_ph_inductively}, one has that $\vp_j^{p, q}$ is pluriharmonic for any $(p, q)$ with $p+q=2m-1$. 
Therefore it follows that 
\begin{equation}\label{eq:thm12_lasteq}
 \sum_{\lambda, \mu=1}^d(\vp_j)_{z_j^\lambda \overline{z_j^\mu}}\cdot \xi^\lambda\cdot \overline{\xi^\mu}  = \sum_{p+q = 2m}\left(\sum_{\lambda, \mu=1}^d(\vp_j^{p, q})_{z_j^\lambda\overline{z_j^\mu}}\cdot \xi^\lambda\cdot \overline{\xi^\mu} \right)\cdot w_j^p\overline{w_j}^q + o(|w_j|^{2m}). 
\end{equation}

\begin{proof}[Proof of {Lemma \ref{lem:12main}}]
Let us show the lemma by induction. 
Note that the assertion has already been shown when $\mu=0$ in the previous subsection. 
Assume that the assertion holds for $(p, q)=(0, m), (1, m), \dots, (\mu-1, m)$, and $(p, q)=(m, 0), (m, 1), \dots, (m, \mu-1)$. 
For a pair $(p, q)$ with $m+1\leq p+q\leq 2m$, let 
\begin{align*}
\vp_k^{p, q}(z_k) &= \vp_k^{p, q}(z_k(z_j, w_j))\\ 
&= \vp_k^{p, q}(z_k(z_j, 0)) + \sum_{0<\nu+\lambda\leq 2m+2-(p+q)} H_{kj, \nu\lambda}^{p, q}(z_j)\cdot w_j^\nu\overline{w_k}^\lambda + o\left(|w_j|^{2m+2-(p+q)}\right). 
\end{align*}
be the expansion of the function $\vp_k^{p, q}(z_k(z_j, w_j))$ of class $C^{2m+2-(p+q)}$. 
Then the first term of the right hand side of the equation (\ref{eq:diff_vp's_4}) can be described as 
\begin{align*}
&\sum_{m+1\leq p+q\leq 2m}(t_{jk}^{-p+q}\vp_k^{p, q}(z_k)-\vp_j^{p, q}(z_j))\cdot w_j^p\overline{w_j}^q \\
&= \sum_{m+1\leq p+q\leq 2m}(t_{jk}^{-p+q}\vp_k^{p, q}(z_k(z_j, 0))-\vp_j^{p, q}(z_j))\cdot w_j^p\overline{w_j}^q \\
&\ \ \ \ +\sum_{m+1\leq p+q\leq 2m}\ \ \ \sum_{0<\nu+\lambda\leq 2m+2-(p+q)} t_{jk}^{-p+q}H_{kj, \nu\lambda}^{p, q}(z_j)\cdot w_j^{p+\nu}\overline{w_j}^{q+\lambda} \\
&\ \ \ \ + o(|w_j|^{2m+2}). 
\end{align*}
As the functions $F_{kj, m+2}$, $G_{kj, m+2}^{(p)}$, and $G_{kj, m+2}^{(q)}$ are holomorphic functions with zero of multiplicity larger than $m$ along $w_j=0$, 
it follows from the equation (\ref{eq:diff_vp's_4}) that 
\[
t_{jk}^{-m+\mu}\vp_k^{m, \mu}(z_k(z_j, 0))-\vp_j^{m, \mu}(z_j)\equiv 0
\]
holds (Here we used the fact that $H^{p, q}_{kj, \nu\lambda}\equiv 0$ holds if $p\leq m$ and $q< \mu (\leq m)$, which follows from (Assertion)$_{m-1}$ and the inductive assumption). 

When $\mu<m$, one has that $m+\nu\leq 2m-1$. 
In this case, it follows from the equation 
(\ref{eq:thm12_lasteq}) and the argument as in the proof of 
Lemma \ref{lem_ph_inductively} that $\{(U_j, \vp_j^{m, \mu})\}$ defines a global section of $N_{Y/X}^{-m+\mu}$ which is pluriharmonic. 
Therefore $\{(U_j, \vp_j^{m, \mu})\}$ is a locally constant global section. 
One can show that $\{(U_j, \vp_j^{\mu, m})\}$ is a locally constant global section of $N_{Y/X}^{m-\mu}$ in the same manner. 

Finally, let us consider the functions $\vp_j^{m, m}$'s. 
By the same argument as above, one has that 
$\{(U_j, \vp_j^{m, m})\}$ glues up to define a global function on $Y$. 
As the function $\vp_j$ is plurisubharmonic, 
it follows from Lemma \ref{lem_positive_function} and the equation (\ref{eq:thm12_lasteq}) 
that $\vp_j^{m, m}$ is also plurisubharmonic. 
Therefore, as $Y$ is compact, one has that $\vp_j^{m, m}$ is constant by using the maximum principle. 
\end{proof}

(Assumption)$_m$ follows from Lemma \ref{lem:12main}, which completes the proof of Theorem \ref{thm:submain_1}. 
\qed

\subsection{Proof of Corollary \ref{cor:main}}
It follows form Theorem \ref{thm:submain_1} that the pair $(Y, X)$ is of infinite type when $[Y]$ is $C^\infty$-semipositive. 
Therefore the corollary follows from Ueda's theorem Theorem \ref{thm:ueda_linearization}.  
\qed

\section{Proof of Theorem \ref{thm:submain_2}}

Under both of the assumptions $(a)$ and $(b)$, one can easily show that $(ii)$ implies $(i)$, see \cite[\S 2.1]{K2019}. 
In this section, we assume the assertion $(i)$ and show $(ii)$. 
As the theorem follows from Corollary \ref{cor:main} when $N_{Y/X}^m$ is holomorphically trivial for a positive integer $m$, 
we also assume that $N_{Y/X}^m$ is not holomorphically trivial for any positive integer $m$. 

As $[Y]$ is $C^\infty$-semipositive, one can take a function $\psi$ as in Proposition \ref{prop:1}. 
For a sufficiently large real number $N$, define a domain $\Omega_N$ by 
$\Omega_N\coloneqq \{x\in X\setminus Y\mid \psi(x)<N\}$ 
under the assumption $(a)$, and by 
$\Omega_N\coloneqq \{x\in \Omega\setminus Y\mid \psi(x)<N\}$ 
under the assumption $(b)$. 
In what follows, we always assume that $N$ is large enough so that $\del\Omega_N$ is a real submanifold of class $C^\infty$ (see Remark \ref{rmk:omega_boundary}). 

First, we show the following: 

\begin{lemma}\label{lem:Hartogs_does_not_hold}
The Hartogs type extension theorem does not hold on $\Omega_N$ for a sufficiently large $N$. 
\end{lemma}

\begin{proof}
Under the assumption $(b)$, the boundary $\del\Omega_N$ has at least two connected components for a sufficiently large $N$, from which the assertion easily follows. 
Assume that $(a)$ holds, and that there exists a sequence $\{N_\nu\}_{\nu=0}^\infty\subset \mathbb{R}$ such that $N_\nu\to \infty$ as $\nu\to \infty$ and that the Hartogs type extension theorem holds on $\Omega_{N_\nu}$ for any integer $\nu$. 
Take a neighborhood $V$ of $Y$ in $X$. 
As $Y$ is compact, one may assume that $\del V$ is relatively compact in a small neighborhood of $Y$ by shrinking $V$ if necessary. 
Therefore, there exists a positive integer $\nu_0$ such that $\{\psi=N_{\nu_0}\}$ is a relatively compact subset of $V\setminus Y$. 
 
Let $f$ be a holomorphic function defined on $V\setminus Y$. 
As $V\cap \Omega_{N_{\nu_0}}$ is a neighborhood of the boundary $\del \Omega_{N_{\nu_0}}$ in $\Omega_{N_{\nu_0}}$, it follows from the assumption that there exists a holomorphic function $F$ on $\Omega_{N_{\nu_0}}$ such that $F|_{V\cap \Omega_{N_{\nu_0}}}=f|_{V\cap \Omega_{N_{\nu_0}}}$. 
As $\{(V\setminus Y, f), (\Omega_{N_{\nu_0}}, F)\}$ glues up to define a holomorphic function on $X\setminus Y$, the Hartogs type extension theorem holds on $X\setminus Y$, which contradicts to the condition $(a)$. 
\end{proof}

It follows from Ohsawa's theorem Theorem \ref{thm:ohsawa_Hartogs} and 
Lemma \ref{lem:Hartogs_does_not_hold} that $\del\Omega_N$ is Levi-flat for a sufficiently large real number $N$. 
Thus $\del\Omega$ is Levi-flat under the assumption $(b)$, 
and the real hypersurface $\{\psi=N\}$ is Levi-flat for a sufficiently large $N$ under the assumptions $(a)$ and $(b)$. 
From this, we have the existence of a neighborhood $V_0$ of $Y$ such that the complex Hessian of the function $\psi$ has exactly one positive eigenvalue at any point of $V_0\setminus Y$. 
Indeed, as $\psi$ is a function as in Proposition \ref{prop:1} $(ii)$, one can easily show the existence of a neighborhood $V'$ of $Y$ such that the complex Hessian of $\psi$ has at least one positive eigenvalue at any point of $V'\setminus Y$. 
If there exists a sequence $\{p_\nu\}_{\nu=0}^\infty$ of points in $V'\setminus Y$ such that $p_\nu$ accumulates on  $Y$ and that the complex Hessian of $\psi$ has at least two positive eigenvalues at each $p_\nu$, one has that the hypersurface $\{x\in V'\setminus Y\mid \psi(x)=\psi(p_\nu)\}$ is not Levi-flat (see \S \ref{section:ohsawa_Hartogs_ext}), which contradicts to Lemma \ref{lem:Hartogs_does_not_hold}. 

Take such a neighborhood $V_0$ of $Y$. 
Then one can consider the Monge--Amp\`ere foliation $\mathcal{F}$ on $V_0\setminus Y$ associated with $\sqrt{-1}\ddbar \psi$, which is of real-codimension $2$ and whose leaves are the images of holomorphic immersions (see \cite{S} and \cite[Theorem 2.4]{BK}). 
For this foliation $\mathcal{F}$, we first prove the following: 

\begin{lemma}\label{lem:hojyo}\ \\
$(i)$ For a sufficiently large positive number $N$, any leaf of the Levi foliation of $\{\psi = N\}$ is a leaf of $\mathcal{F}$. \\
$(ii)$ The foliation $\mathcal{F}$ is holomorphic also in the transversal directions by shrinking $V_0$ if necessary. \\
$(iii)$ There exists a holomorphic (non-singular) foliation on $V_0$ which has $Y$ as a leaf and which coincides with $\mathcal{F}$ on $V_0\setminus Y$. 
\end{lemma}

\begin{proof}
$(i)$ Let $N$ be a sufficiently large number and $L$ be a leaf of the Levi foliation of $\{\psi = N\}$. 
Note that $L$ is the image of a holomorphic immersion from some complex manifold to $X$. 
Then, as it is clear that $\psi|_L$ is a constant function, the assertion $(i)$ holds. 

$(ii)$ By using the dual of $(\del\psi, d z_j^1, \dots, dz_j^{n-1})$ as a basis of the tangent bundle at each point of $(V_0\cap V_j)\setminus Y$ ($\{V_j\}$ is as in \S \ref{section:ueda_theory} and $z_j=(z_j^1, z_j^2, \dots, z_j^{n-1})$ is a holomorphic coordinate of the leaf which passes through the point) and considering the representation matrix of the complex Hessian of $\psi$, 
it follows form Lemma \ref{lem:matrix_fundamental} below that 
there exists a function $F\colon V_0\setminus Y\to \mathbb{R}_{>0}$ such that 
\[
\ddbar \psi = F\cdot \del\psi\wedge \delbar \psi
\]
holds. 
As the left hand side of the equation above is $\del$-closed, one has that 
\[
0 = (\del F)\wedge \del\psi\wedge \delbar \psi
-F\cdot \del\psi \wedge \ddbar \psi
= (\del F)\wedge \del\psi\wedge \delbar \psi
-F^2\cdot \del\psi \wedge \del\psi\wedge \delbar \psi, 
\]
from which it follows that $(\del F)\wedge \del\psi\wedge \delbar \psi = 0$. 
Similarly, it follows that  $(\delbar F)\wedge \del\psi\wedge \delbar \psi = 0$. 
Thus one has that $(dF)\wedge \del\psi\wedge \delbar \psi = 0$, 
which implies that the function $F$ is leafwise constant. 

Let us show that $F$ is constant on $\{\psi=N\}$ if $N$ is sufficiently large by contradiction. 
Assume that there exists an increasing sequence $\{N_\nu\}_{\nu=0}^\infty$ of positive numbers such that $N_\nu\to \infty$ as $\nu\to \infty$ and that $F$ is not constant on each $\{\psi=N_\nu\}$. 
Take $\nu$ sufficiently large so that $\{\psi>N_\nu\}\cup Y$ is a sufficiently small tubular neighborhood of $Y$ whose boundary is a real submanifold of class $C^\infty$  and is a circle bundle over $Y$ (see Remark \ref{rmk:omega_boundary}). 
Fix a (small) real $2$-dimensional disc $D$ which intersect $Y$ transversally. 
Denote by $S_\nu$ the intersection $D\cap \{\psi=N_\nu\}$. 
We may assume that $S_\nu$ is diffeomorphic to a circle. 
Denote by $h_\nu\colon \pi_1(Y, *)\to {\rm Diff}(S_\nu)$ the holonomy of the Levi foliation on $\{\psi=N_\nu\}$, where $*$ is a base point of $Y$ and ${\rm Diff}(S_\nu)$ is the diffeomorphism group of $S_\nu$ (For the definition of the holonomy map, see \cite{KO} for example). 
As the image $F(\{\psi=N_\nu\})$ is not a point and $F$ is leafwise constant, one has that $F(S_\nu)$ is not a point. 
Thus it follows from Sard's theorem that there exists a real number $m_\nu\in F(S_\nu)$ such that $(F|_{S_\nu})^{-1}(m_\nu)$ is a finite subset of $S_\nu$. 
Take a point $p_\nu\in (F|_{S_\nu})^{-1}(m_\nu)$ and denote by $L_\nu$ the leaf of the Levi foliation on $\{\psi=N_\nu\}$ such that $p_\nu\in L_\nu$. 
As $F$ is leafwise constant, one has that $F(h_\nu(\gamma)(p_\nu))=F(p_\nu)$ for any $\gamma\in \pi_1(Y, *)$, 
which implies that $\{h_\nu(\gamma)(p_\nu)\mid \gamma\in\pi_1(Y, *)\}$ is a finite subset of $S_\nu$. 
Therefore, $L_\nu$ is a compact holomorphic submanifold of $X\setminus Y$. 
Additionally one has that $L_\nu$ is homologous to a cycle $\mu_\nu Y$ of $X$ for some positive integer $\mu_\nu$, 
since $L_\nu$ is sufficiently close to $Y$ and it topologically covers $Y$ (consider the restriction of the projection map of the topological circle bundle $\{\psi=N_\nu\}\to Y$). 
By \cite[Proposition 2.7]{CLPT}, we may assume that the natural map $H^1(X, \mathcal{O}_X)\to H^1(Y, \mathcal{O}_Y)$ is injective, since otherwise the natural map $X\to {\rm Alb}(X)/{\rm Alb}(Y)$ defines a fibration such that $Y$ is a fiber, which contradicts to our assumption that $N_{Y/X}$ is a non-torsion element of the Picard variety ${\rm Pic}^0(Y)$. 
As the natural map ${\rm Pic}^0(X)\to {\rm Pic}^0(Y)$ is a finite covering to the image in this case, for sufficiently large $\nu_1$ and $\nu_2$, 
there exists a positive integer $M$ such that 
the line bundle $[\mu_{1} Y-L_{\nu_1}]^{\otimes M\mu_2}$ coincides with $[\mu_{2} Y-L_{\nu_2}]^{\otimes M\mu_1}$, where we are denoting $\mu_{\nu_\lambda}$ by $\mu_{\lambda}$ ($\lambda=1, 2$), 
since both $[\mu_{1} Y-L_{\nu_1}]^{\otimes \mu_2}$ and $[\mu_{2} Y-L_{\nu_2}]^{\otimes \mu_1}$ are topologically trivial line bundles whose restrictions to $Y$ coincide with the unitary flat line bundle $N_{Y/X}^{\mu_1\mu_2}$. 
Thus, as the line bundle $[M\mu_2 L_{\nu_1}-M\mu_1 L_{\nu_2}]$ is holomorphically trivial, 
one has that there exists a non-constant holomorphic function on a neighborhood $X\setminus (L_{\nu_1}\cup L_{\nu_2})$ of $Y$, which contradicts to the assumption that $N_{Y/X}$ is a non-torsion element of ${\rm Pic}^0(Y)$ (see also \cite[Corollary (2.4)]{Su} for example). 

Therefore, by shrinking $V_0$ if necessary, we may assume that there exists a function $G\colon \mathbb{R}\to \mathbb{R}$ such that $F=G\circ \psi$. 
Fix a point $p\in V_0\setminus Y$ and take a function $\chi\colon\mathbb{R}\to \mathbb{R}$ such that 
\[
\chi'(t) = \exp\left(-\int_{\psi(p)}^t G(s)\,ds\right). 
\]
As $G(t)\cdot \chi'(t)+\chi''(t)\equiv 0$, the function $\widehat{\psi}\coloneqq \chi\circ \psi$ satisfies 
\[
\ddbar\widehat{\psi} 
= \chi'(\psi)\cdot \ddbar \psi + \chi''(\psi)\cdot \del\psi\wedge \delbar \psi
= \chi'(\psi)\cdot \left(\ddbar \psi - G\circ\psi\cdot \del\psi\wedge \delbar \psi\right)\equiv 0, 
\]
which means that $\widehat{\psi}$ is pluriharmonic. 
The assertion $(ii)$ follows from this, since contour hypersurfaces of $\widehat{\psi}$ are also contour hypersurfaces of $\psi$. 

$(iii)$ 
Take an open covering $\{U_j\}$ of $Y$, a neighborhood $V_j$ of each $U_j$, and coordinates $(z_j, w_j)$ of $V_j$ as in \S \ref{section:ueda_theory}. 
From Theorem \ref{thm:submain_1}, it follows that we may assume that the system $\{(V_j, w_j)\}$ is of type $3$. 
In this case, by (Assertion)$_2$ in the previous section, there exist constants $A_j^{(1)}$, $A_j^{(2)}$, and $B$ such that 
\[
\psi = -\log |w_j|^2 + A_j^{(1)}w_j+\overline{A_j^{(1)}}\overline{w_j} + A_j^{(2)}w_j^2+\overline{A_j^{(2)}}\overline{w_j}^2 + B\cdot |w_j|^2 + O(|w_j|^3)
\]
holds as $|w_j|\to 0$ on each $V_j$. 
Thus one has the description
\begin{equation}\label{eq:ddbar_vp_foliation}
\sqrt{-1}\ddbar \vp_j=B\cdot \sqrt{-1}dw_j\wedge d\overline{w_j} + O(|w_j|). 
\end{equation}
Denote by $\mathbb{P}(T_{V_0\setminus Y}^*)$ the projective space bundle over $V_0\setminus Y$ whose fibers consists of the lines in the cotangent spaces. 
Let $\sigma$ be the section of $\mathbb{P}(T_{V_0\setminus Y}^*)$ defined by 
\[
p\mapsto \left[a(p)dw_j+\sum_{\lambda=1}^{n-1}b_\lambda(p)\cdot dz_j^\lambda\right], 
\]
where $a(p)dw_j+\sum_{\lambda=1}^{n-1}b_\lambda(p)\cdot dz_j^\lambda$ is a $1$-form which is orthogonal to the leaf passes through $p$. 
From the assertion $(ii)$, it follows that $\sigma$ is a holomorphic section. 
As $B>0$ holds by the construction, it follows from the equation (\ref{eq:ddbar_vp_foliation}) that 
we may assume $a(p)\not= 0$ on a neighborhood of $Y$. 
Therefore there exist holomorphic functions $\widehat{b}_\lambda$'s such that 
\[
\sigma(p) = \left[dw_j+\sum_{\lambda=1}^{n-1}\widehat{b}_\lambda(p)\cdot dz_j^\lambda\right]. 
\]
Again by the equation (\ref{eq:ddbar_vp_foliation}), one has that $\widehat{b}_\lambda(p)\to 0$ as $p\to Y$, 
from which it follows that $\sigma$ can be holomorphically extended to $V_0$ by letting $\sigma(p)\coloneqq [dw_j]$ if $p\in Y$ (Apply Riemann's extension theorem to each $\widehat{b}_\lambda$). 
As ${\rm Ker}(\sigma)$ is clearly  involutive (and thus integrable), it defines a holomorphic extension of the foliation $\mathcal{F}$. 
\end{proof}

\begin{lemma}\label{lem:matrix_fundamental}
Let $A=(a_{\lambda}^\mu)_{1\leq \lambda, \mu\leq n}$ be a positive semi-definite Hermitian matrix of order $n$. 
Assume that $A$ admits exactly one positive eigenvalue, and that 
$a_{\lambda}^\mu=0$ holds if $2\leq \lambda\leq n$ and $2\leq \mu\leq n$. 
Then $a_1^\lambda=0$ and $a_\lambda^1=0$ hold for each $\lambda\in\{2, 3, \dots, n\}$. 
\end{lemma}

\begin{proof}
For each $\lambda\in\{2, 3, \dots, n\}$, consider the matrix 
\[
A_\lambda\coloneqq\left(
\begin{array}{rr}
a_1^1 & a_1^\lambda \\
a_\lambda^1& a_\lambda^\lambda \\
\end{array}
\right). 
\]
Then it holds that ${\rm det}\,A_\lambda=-|a_\lambda^1|^2$, since $a_\lambda^\lambda=0$. 
As $A_\lambda$ has zero as an eigenvalue, one has that ${\rm det}\,A_\lambda=0$, from which the lemma follows. 
\end{proof}

In what follows, 
we use the notation in the proof of Lemma \ref{lem:hojyo} $(iii)$, 
and denote the extended foliation on $V_0$ as in Lemma \ref{lem:hojyo} $(iii)$ by the same letter $\mathcal{F}$. 
Next, we prove the following lemma and proposition. 
 
\begin{lemma}\label{lem:last1}
By shrinking $V_j$'s if necessary, there exists a holomorphic function $\eta_j$ on each $V_j$ such that $\{(V_j, (z_j, \eta_j))\}$ is a foliation chart of $\mathcal{F}$ (i.e. $(z_j, \eta_j)$ can be regarded as coordinates on $V_j$ and each leaf can be described as $\{\eta_j=\text{constant}\}$ in $V_j$), and that 
$\eta_j=w_j+O(w_j^2)$ holds as $w_j\to 0$. 
\end{lemma}

\begin{proof}
As $\mathcal{F}$ is a holomorphic foliation, there exists a holomorphic foliation chart $\{(V_j, (z_j, \eta_j))\}$ by shrinking $V_j$'s if necessary. 
Let 
\[
\eta_j=\sum_{\nu=0}^\infty a_{j, \nu}(z_j)\cdot w_j^\nu
\]
be the expansion. 
As $Y$ is a leaf of $\mathcal{F}$, we may assume that $a_{j, 0}\equiv 0$ for each $j$. 

By definition of the foliation chart, the restriction $\vp_j|_{\{\eta_j=c\}}$ is pluriharmonic for each constant $c$. 
Thus, for each point $p$ in $V_0\cap V_j$, one has that there exists a positive number $F_j(p)$ such that 
\begin{equation}\label{eq:eta_MA}
(\del\eta_j\wedge \delbar\overline{\eta_j})_p=F_j(p)\cdot (\ddbar \vp_j)_p 
\end{equation}
holds (Apply Lemma \ref{lem:matrix_fundamental} to the representation matrix of $(\sqrt{-1}\ddbar \vp_j)_p $ with respect to the basis $\del/\del \eta_j$, $\del/\del z_j^1$, \dots, $\del/\del z_j^d$ of $T_{V_0, p}$). 
By comparing the leading terms of the equation (\ref{eq:eta_MA}), one easily has that 
\[
F_j|_{U_j} = \frac{|a_{j, 1}|^2}{B}
\]
holds. 
By taking the exterior derivative of the equation (\ref{eq:eta_MA}), one has that $dF_j\wedge \ddbar \vp_j \equiv 0$, 
since both $\del\eta_j\wedge \delbar\overline{\eta_j}$ and $\ddbar \vp_j$ are $d$-closed. 
From this, it follows that $d(F_j|_{U_j})\equiv 0$ holds. 
Therefore one has that $a_{j, 1}$ is a constant function, from which the lemma follows. 
\end{proof}

\begin{proposition}\label{prop:last} 
It holds that $({\rm Hol}_{\mathcal{F}, Y}(\gamma))'(0)\in\U$ for any 
element $\gamma\in \pi_1(Y, *)$, where 
\[
{\rm Hol}_{\mathcal{F}, Y}\colon \pi_1(Y, *)\to {\rm Diff}(\mathbb{C}, 0)\coloneqq \{f\in \mathcal{O}_{\mathbb{C}, 0}\mid f(0)=0, f'(0)\not=0\}
\]
is the holonomy of $\mathcal{F}$ along $Y$. 
\end{proposition}
%

\begin{proof}
The assertion follows from Lemma \ref{lem:last1} and the definition of the holonomy. 
\end{proof}

As we are assuming that $N_{Y/X}^m$ is not holomorphically trivial for any positive integer $m$, there exists an element $\gamma\in \pi_1(Y, *)$ such that $f_\gamma\coloneqq {\rm Hol}_{\mathcal{F}, Y}(\gamma)$ satisfies $(f_\gamma'(0))^m\not=1$ for any $m\in\mathbb{Z}_{>0}$, 
since the map $\gamma\mapsto f_\gamma'(0)$ coincides with the monodromy representation $\rho_{N_{Y/X}}\colon \pi_1(Y, *)\to \U$ of the unitary flat line bundle $N_{Y/X}$.

Let us show that $f_\gamma$ is linearizable: i.e. there exists a constant $\lambda\in \U$ such that $f_\gamma(w)=\lambda w$ holds by using a suitable coordinate $w$ of a transversal of $\mathcal{F}$. 
Assume that $f_\gamma$ is not linearizable. 
Then it follows from  Proposition \ref{prop:last} and \cite[Theorem IV.2.3]{P} that, for any neighborhood $U$ of the origin in $\mathbb{C}$, there exists a point $w\in U\setminus\{0\}$ such that $f_\gamma^m(w)\in U\setminus \{0\}$ holds for any integer $m$ and $0\in \overline{\{f_\gamma^m(w)\mid m\in \mathbb{Z}\}}$. 
Thus one has that, for any real number $N$, there exists a leaf $L$ of $\mathcal{F}$ such that $L\cap Y=\emptyset$, $L\subset \{\psi>N\}\cup Y$, and that $\overline{L}\cap Y\not=\emptyset$ (see also \cite[\S 4]{KO}), which clearly contradicts to Lemma \ref{lem:hojyo} $(i)$. 

Therefore, 
one can take a $1$-dimensional disc $D\subset V_0$ which transversally intersects $Y$ at the base point $*$ and a coordinate $w$ of $D$ such that $f_\gamma(w)=\lambda w$ for some constant $\lambda\in \U$ (the base point corresponds to the origin: i.e. $w(*)=0$). 
Take a point $p_0$ in $D$ sufficiently close to the base point. 
Let $w_0$ be a complex number which corresponds to $p_0$ (i.e. $w_0\coloneqq w(p_0)$) 
and $L_0$ be the leaf of $\mathcal{F}$ which contains $p_0$. 
Set $\ve_0 \coloneqq  |w_0|$ and $N_0 \coloneqq \psi(p_0)$. 
Then, as $\overline{\{\lambda^mw_0\mid m\in\mathbb{Z}\}}=\{w\in \mathbb{C}\mid |w|=\ve_0\}$, one has that 
$L_0\cap D$ is a dense subset of $\{w\in \mathbb{C}\mid |w|=\ve_0\}$. 
Thus one has that
\begin{equation}\label{eq:foliation_S1}
\{\psi = N_0\}\cap D=\{p\in D \mid |w(p)|=\ve_0\}, 
\end{equation}
because it follows from by Lemma \ref{lem:hojyo} $(i)$ that $\overline{L_0}\subset \{\psi = N_0\}$, and $\{\psi=N_0\}\cap D$ is diffeomorphic to a circle if $N_0$ is enough large (see Remark \ref{rmk:omega_boundary}). 
Take an element $\gamma'\in \pi_1(Y, *)$. 
Then it follows from (\ref{eq:foliation_S1}) that 
$f_{\gamma'}(w_0)\in \{w\in \mathbb{C}\mid |w|=\ve_0\}$, since there exists a point $p\in L_0\cap D (\subset \{\psi = N_0\}\cap D)$ such that $w(p) = f_{\gamma'}(w_0)$ by definition of the holonomy. 
Thus one has that $f_{\gamma'}$ can be regarded as an element of the automorphism group of $\{w\in\mathbb{C}\mid |w|<\delta\}$ for a sufficiently small positive number $\delta$, which implies that there exists a constant $\lambda_{\gamma'}\in \U$ such that $f_{\gamma'}(w)=\lambda_{\gamma'}w$. 

Take an index $j$ such that $*\in U_j$. 
Define a holomorphic function $\zeta_j$ on $V_j$ so that $\zeta_j|_{V_j\cap D}=w$ and that $\zeta_j$ is leafwise constant. 
On $V_k$ with $V_{jk}\not=\emptyset$, define a leafwise constant holomorphic function $\zeta_k$ on $V_k$ so that $\zeta_k|_{V_{jk}}=\zeta_j|_{V_{jk}}$. 
Inductively, one can define a leafwise constant holomorphic function $\zeta_k$ on each $V_k$ in the same manner. 
As the degrees of freedom of this definition of each function $\zeta_k$ can be described by using the holonomy functions, which have been shown to be $\U$-linear, one has that there exist constants $s_{jk}\in \U$ such that $\zeta_j=s_{jk}\zeta_k$ on each $V_{jk}$. 
As these constants $s_{jk}$'s can be regarded as the transition functions of $[Y]|_{V}$ for a small neighborhood $V$ of $Y$, $[Y]|_{V}$ is unitary flat. 
\qed

%

\section{Proof of Theorem \ref{thm:main}}

\subsection{Proof of Theorem \ref{thm:main} $(i)$}
By Theorem \ref{thm:submain_1}, $(Y, X)$ is of infinite type. 
Thus Theorem \ref{thm:main} $(i)$ follows from \cite[Theorem 5.1]{N} (see also \cite[Theorem 1.3]{CLPT}). 
\qed

\begin{remark}\label{rmk:mainthm-1}
\cite[Theorem 5.1]{N} and \cite[Theorem 1.3]{CLPT} only require the condition that Ueda type of $(Y, X)$ is larger than $m$ for implying the existence of the fibration as in Theorem \ref{thm:main} $(i)$. 
Thus Theorem \ref{thm:main} $(i)$ holds not only when $[Y]$ is $C^\infty$-semipositive but also $[Y]$ is $C^r$-semipositive if $r\geq \max\{3, 2m\}$. 
\end{remark}

From Remark \ref{rmk:mainthm-1}, it is easily observed that, when $X$ is compact K\"ahler and $N_{Y/X}^m$ is holomorphically trivial, 
$[Y]$ is $C^\omega$-semipositive (i.e. $[Y]$ admits a real-analytic Hermitian metric with semipositive curvature) if and only if it is $C^r$-semipositive for some integer $r$ with $r\geq \max\{3, 2m\}$ 
(Consider the pull-back of the Fubini--Study type metric). 
In consideration of this fact, it seems to be natural to pose the following: 

\begin{question}
Does there exist an example $(X, Y)$ of a complex manifold $X$ and a hypersurface $Y\subset X$ with unitary flat normal bundle  $[Y]|_Y$ such that, for some non-negative integer $r$, 
$[Y]$ is $C^r$-semipositive and is not $C^{r+1}$-semipositive?
\end{question}

Note that there exist a projective surface $X$ and a non-singular curve $Y$ of $X$ with unitary flat normal bundle such that $[Y]$ is $C^\infty$-semipositive whereas it is not $C^\omega$-semipositive, see \cite[Theorem 1]{B} (see also \cite[Theorem 3]{U} for Diophantine normal bundles). 

\begin{remark}
One can also show Theorem \ref{thm:main} $(i)$ without using \cite[Theorem 5.1]{N} itself, 
but running an argument along the same strategy as the proof of \cite[Theorem 1.3]{CLPT}. 
Indeed, as it follows from \cite[Proposition 2.7]{CLPT} that 
the fibration as in the theorem exists if the natural map $H^1(X, \mathcal{O}_X)\to H^1(Y, \mathcal{O}_Y)$ is not injective. 
Therefore it is sufficient to consider the case where 
${\rm Pic}^0(X)\to {\rm Pic}^0(Y)$ is a finite covering map onto the image. 
Let $f\colon V\to \mathbb{C}$ be as in Corollary \ref{cor:main}. 
Take a point $p\in \mathbb{C}$ sufficiently close to the origin. 
Then, for the fiber $D\coloneqq f^{-1}(p)$, the line bundle $[mY-D]^{\otimes M}$ on $X$ is holomorphically trivial for some positive integer $M$, 
since $[mY-D]$ is topologically trivial and $[mY-D]|_Y$ is holomorphically trivial ($M$ is the degree of the finite covering ${\rm Pic}^0(X)\to {\rm Image}({\rm Pic}^0(X)\to {\rm Pic}^0(Y))$). 
Thus one has that $[MD]=[mMY]$. 
Let $f$ be a holomorphic global section of $[MD]$ whose zero divisor coincides with $MD$, 
and $g$ be a holomorphic global section of $[MD]$ whose zero divisor coincides with $mMY$. 
The theorem follows by considering the map $X\ni x\mapsto [f(x); g(x)]\in \mathbb{P}^1$. 
\end{remark}

\subsection{Proof of Theorem \ref{thm:main} $(ii)$}
Assume that the Hartogs type extension theorem does not hold on $X\setminus Y$. 
Then, as the assertion $(a)$ of 
Theorem \ref{thm:submain_2} holds, 
it follows from Theorem \ref{thm:submain_2} that there exists a neighborhood $V$ of $Y$ such that $[Y]|_{V}$ is unitary flat. 
As $N_{Y/X}^m$ is not holomorphically trivial for any positive integer $m$, 
one has that, for any neighborhood $\Omega$ of $Y$, there exists a compact Levi-flat hypersurface $H$ in $\Omega\setminus Y$ such that each leaf of the Levi foliation of $H$ is dense in $H$. 
Then it holds that any holomorphic function on $\Omega\setminus Y$ is constant (see the proof of \cite[Lemma 2.2]{KU} for example). 
Therefore it is clear that the Hartogs type extension theorem holds on $X\setminus Y$. 
\qed

\subsection{Proof of Theorem \ref{thm:main_last}}
By considering the flat metric, one can easily show that the assertion $(ii)$ implies $(i)$. 
Thus here we show that the assertion $(i)$ implies $(ii)$. 

Assume that $[Y]$ is $C^\omega$-semipositive on a neighborhood of $Y$. 
Then, by running the same argument as in the proof of Proposition \ref{prop:1}, 
one has a plurisubharmonic function $\psi$ on $X\setminus Y$ as in the proposition 
which is real analytic on a neighborhood of $Y$. 
Take a system $\{(V_j, w_j)\}$ of type $3$, whose existence is assured by Theorem \ref{thm:submain_1}. Set $\vp_j(z_j, w_j)\coloneqq \log |w_j|^2+\psi(z_j, w_j)$, where $z_j$ is as in \S \ref{section:ueda_theory}. 
By shrinking $V_j$ if necessary, we may assume that $\vp_j$ is a real-analytic plurisubharmonic function on $V_j$. 
Denote by $g_j$ the local Euclidean metric of $V_j$ with respect to $(z_j, w_j)$: 
i.e. the fundamental form of $g_j$ is $\sqrt{-1}dw_j\wedge d\overline{w_j} + \sqrt{-1}\sum_{\lambda=1}^{n-1}dz_j^\lambda\wedge d\overline{z_j^\lambda}$. 
Then the function $F_j$ on $V_j$ defined by 
\[
F_j\coloneqq \frac{|\ddbar \psi|_{g_j}}{|\del\psi\wedge \delbar\psi|_{g_j}}
=\frac{|w_j|^2\cdot |\ddbar \vp_j|_{g_j}}{|(dw_j-w_j\del\vp_j)\wedge (d\overline{w_j}-\overline{w_j}\delbar\vp_j)|_{g_j}}
\]
is clearly a function of class $C^\omega$ (by shrinking $V_j$ if necessary). 

Let us first show that 
\begin{equation}\label{eq:Leviflat_C_omega}
\ddbar \psi = F_j\cdot \del\psi\wedge \delbar\psi 
\end{equation}
holds on a neighborhood of $U_j$. 
For that purpose, consider the difference $\rho_j\coloneqq \ddbar \psi - F_j\cdot \del\psi\wedge \delbar\psi$ on $V_j$. 
Take a positive integer $m$. 
Then, from (Assertion)$_{m+2}$ in the proof of Theorem \ref{thm:submain_1}, it follows that, if one shrink $V_j$, there exists a holomorphic defining function $\widehat{w}_j$ of $U_j$ in $V_j$ such that 
\[
\vp_j(z_j, \widehat{w}_j) = \sum_{p+q\leq m+2}A_j^{p, q}\widehat{w}_j^p\overline{\widehat{w}_j}^q + O(|w_j|^{m+3})
\]
holds for some constants $A_j^{p, q}$'s as $|w_j|\to 0$. 
Thus one has that 
\[
\ddbar\vp_j = \left(\sum_{p+q\leq m}(p+1)(q+1)A_j^{p+1, q+1}\widehat{w}_j^p\overline{\widehat{w}_j}^q\right) d\widehat{w}_j\wedge d\overline{\widehat{w}_j} + O(|w_j|^{m+1}) 
\]
and 
\[
\del\vp_j= \left(\sum_{p+q\leq m}(p+1)A_j^{p, q}\widehat{w}_j^{+1}p\overline{w_j}^q\right)d\widehat{w}_j + O(|w_j|^{m+1})
\]
hold, from which it follows that $\rho_j = O(|w_j|^{m+1})$ as $|w_j|\to 0$. 
As $m$ is arbitrary and $\rho_j$ is real analytic, one have that $\rho_j\equiv 0$ on a neighborhood of $Y$. 

From the equation (\ref{eq:Leviflat_C_omega}), 
it directly follows that, for a sufficiently large $N$ and a point $p\in \{\psi=N\}$, 
$(\sqrt{-1}\ddbar\psi)_p (v\wedge\overline{w})\equiv 0$ for $v, w\in T_{X, p}$ if $v$ and $w$ are orthogonal to $(\del\psi)_p$. 
Thus $\{\psi=N\}$ is Levi-flat for a sufficiently large $N$. 
Therefore, the theorem follows from the same argument as in the proof of Theorem \ref{thm:submain_2}. 
\qed




\begin{thebibliography}{99}
 \bibitem[BK]{BK} \textsc{E. Bedford, M. Kalka}, Foliations and complex monge--amp\`ere equations, {\bf 30}, 5 (1977), 543--571. 
\bibitem[B]{B} \textsc{M. Brunella}, On K\"ahler surfaces with semipositive Ricci curvature, Riv. Mat. Univ. Parma, {\bf 1} (2010), 441--450. 
\bibitem[CLPT]{CLPT} \textsc{B. Claudon, F. Loray, J.V. Pereira, and F. Touzet}, Compact leaves of codimension one holomorphic foliationson projective manifolds, Ann. Scient. \'Ec. Norm. Sup. (4) {\bf 51} (2018), 1389--1398. 
 \bibitem[D]{D} \textsc{J.-P. Demailly}, Complex Analytic and Differential Geometry, monograph, 2012, available at http://www-fourier.ujf-grenoble.fr/~demailly.
\bibitem[K1]{K2015} \textsc{T. Koike}, On the minimality of canonically attached singular Hermitian metrics on certain nef line bundles, Kyoto J. Math., {\bf 55}, 3 (2015), 607--616. 
\bibitem[K2]{K2018} \textsc{T. Koike}, Higher codimensional Ueda theory for a compact submanifold with unitary flat normal bundle, Nagoya Math. J., DOI: https://doi.org/10.1017/nmj.2018.22. 
\bibitem[KO]{KO} \textsc{T. Koike, N. Ogawa}, On the neighborhood of a torus leaf and dynamics of holomorphic foliations, arXiv:1808.10219.
\bibitem[KU]{KU} \textsc{T. Koike, T. Uehara}, A gluing construction of K3 surfaces, arXiv:1903.01444. 
\bibitem[K3]{K2019} \textsc{T. Koike}, Hermitian metrics on the anti-canonical bundle of the blow-up of the projective plane at nine points, arXiv:1909.06827. 
\bibitem[K4]{K2020} \textsc{T. Koike}, Linearization of transition functions of a semi-positive line bundle along a certain submanifold, arXiv:2002.07830.
\bibitem[N]{N} \textsc{A. Neeman}, Ueda theory: theorems and problems., Mem. Amer. Math. Soc. {\bf 81} (1989), no. 415. 
\bibitem[O1]{O1} \textsc{T. Ohsawa}, 
A remark on pseudoconvex domains with analytic complements in compact K\"ahler manifolds, 
J. Math. Kyoto Univ., {\bf 47}, 1 (2007), 115--119. 
\bibitem[O2]{O2} \textsc{T. Ohsawa}, 
Hartogs type extension theorems on some domains in K\"ahler manifolds, Ann. Polon. Math., {\bf 106}, 1 (2012), 243--254. 
\bibitem[O3]{O3} \textsc{T. Ohsawa}, 
$L^2$ approaches in several complex variables : development of Oka--Cartan theory by $L^2$ estimates for the d-bar operator, Springer, 2015. 
\bibitem[P]{P} \text{R. P\'erez-Marco}, Fixed points and circle maps, Acta Mathematica, {\bf 179} (1997), No.2, 243--294. 
 \bibitem[So]{S} \textsc{S. Sommer},  Komplex-analytische Bl\"atterung reeller Hyperfl\"achen im $C^n$, Math. ann., {\bf 137}, 5 (1959), 392--411. 
\bibitem[Su]{Su} \textsc{O. Suzuki}, Neighborhoods of a compact non-singular algebraic curve imbedded in a $2$-dimensional complex manifold, Publ. Res. Inst. Math. Sci. Kyoto Univ.{\bf 11} (1975), 185--199. 
 \bibitem[U]{U} \textsc{T. Ueda},  On the neighborhood of a compact complex curve with topologically trivial normal bundle, Math. Kyoto Univ., {\bf 22} (1983), 583--607. 
\end{thebibliography}
\end{document}